\documentclass[11pt]{amsart} \usepackage{latexsym, amssymb, amsthm, times, verbatim, amsmath, stmaryrd}
\usepackage[T1]{fontenc}

\newtheorem{thm}{Theorem}[section]

\newtheorem{prop}[thm]{Proposition}
\newtheorem{cor}[thm]{Corollary}
\newtheorem{question}[thm]{Question}
\newtheorem{fact}[thm]{Fact}

\theoremstyle{definition}
\newtheorem{df}[thm]{Definition}

\newtheorem{ex}[thm]{Example}

\newtheorem{conj}[thm]{Conjecture}

\theoremstyle{remark}

\newcommand{\Z}{\mathbb{Z}}

\newcommand{\n}{\mathbb{N}}

\renewcommand{\to}{\rightarrow}

\def \<{\langle}
\def \>{\rangle}

\def \*Z {{{^*}\Z}}
\def \((  {(\!(}
\def \)) {)\!)}

\def \st {\operatorname {st}}

\numberwithin{equation}{section}

\def \BD{\operatorname{BD}}

\makeatletter

\def\indsym#1#2{%
  \setbox0=\hbox{$\m@th#1x$}%
  \kern\wd0%
  \hbox to 0pt{\hss$\m@th#1\mid$\hbox to 0pt{$\m@th#1^{#2}$}\hss}%
  \lower.9\ht0\hbox to 0pt{\hss$\m@th#1\smile$\hss}%
  \kern\wd0}

\def\dotminussym#1#2{%
  \setbox0=\hbox{$\m@th#1-$}%
  \kern.5\wd0%
  \hbox to 0pt{\hss\hbox{$\m@th#1-$}\hss}%
  \raise.6\ht0\hbox to 0pt{\hss$\m@th#1.$\hss}%
  \kern.5\wd0}

\def\nindsym#1#2{%
  \setbox0=\hbox{$\m@th#1x$}%
  \kern\wd0%
  \hbox to 0pt{\hss$\m@th#1\not$\kern1.4\wd0\hss}
  \hbox to 0pt{\hss$\m@th#1\mid$\hbox to 0pt{$\m@th#1^{\,#2}$}\hss}%
  \lower.9\ht0\hbox to 0pt{\hss$\m@th#1\smile$\hss}%
  \kern\wd0}

\allowdisplaybreaks[2]

\newcommand{\starA}{{}^{\ast}A}
\newcommand{\starB}{{}^{\ast}B}
\newcommand{\starN}{{}^{\ast}\n}
\def \FS{\operatorname{FS}}

\title{On supra-SIM sets of natural numbers}
\author{Isaac Goldbring and Steven Leth}
\thanks{Goldbring's work was partially supported by NSF CAREER grant DMS-1349399.}

\address {Department of Mathematics, University of California, Irvine, 340 Rowland Hall (Bldg.\# 400), Irvine, CA, 92697-3875.}
\email{isaac@math.uci.edu}
\urladdr{http://www.math.uci.edu/~isaac}

\address {School of Mathematical Sciences, University of Northern Colorado, Campus Box 122, 501 20th Street, Greeley, Co 80639}
\email{steven.leth@unco.edu}
\urladdr{http://www.unco.edu/nhs/mathematical-sciences/faculty/leth.aspx}

\begin{document}

\begin{abstract}
We introduce the class of supra-SIM sets of natural numbers.  We prove that this class is partition regular and closed under finite-embeddability.  We also prove some results on sumsets and SIM sets motivated by their positive Banach density analogues.     
\end{abstract}

\maketitle

\section{Introduction}

Ramsey theory on the integers can crudely be described as the study of \emph{partition regular} properties of the integers, namely those properties $\mathcal{P}$ of integers such that, whenever $A\subseteq \n$ has $\mathcal{P}$ and $A=B\sqcup C$ (disjoint union), then at least one of $B$ or $C$ has property $\mathcal{P}$.  Here are some of the more prominent examples of partition regular properties of the integers:

\begin{itemize}
\item having infinite cardinality (Pigeonhole principle);
\item having arbitrary long arithmetic progressions (van der Waerden's theorem);
\item containing a set of the form $$\FS(X):=\{a_1+\cdots +a_n \ : \ a_1,\ldots,a_n\in X \text{ distinct}, n\in\n\}$$ for some infinite set $X$ (Hindman's theorem);
\item being piecewise syndetic;
\item having positive Banach density.
\end{itemize}

In this paper, we introduce a new partition regular property of the natural numbers, namely that of being \emph{supra-SIM}.  SIM\footnote{SIM stands for the \emph{standard interval measure} property.} sets were introduced by the second author in \cite{Leth} in connection with  
Stewart and Tijdeman's result that intersections of difference sets of sets of positive density are syndetic.  This property arises from an analogous natural property of \emph{internal} subsets of the nonstandard natural numbers $\starN$ in the sense of nonstandard analysis.  While one can prove an analog of the aforementioned result of Stewart and Tijdeman by replacing the hypothesis of positive Banach density with the assumption of SIM, it was pointed out that the SIM property has some unusual features that should not lead one to view it simply as a notion of largeness.  In particular, it was shown that a SIM set $A$ has the property that all of its supersets are also SIM precisely when $A$ is syndetic.

Thus, it is natural to consider the class of \emph{supra-SIM} sets, which we define to be the class of sets which contain a SIM set.  In this article, we show that the class of supra-SIM sets has better combinatorial features than the class of SIM sets itself.  In particular, we show that this class is partition regular and is closed under \emph{finite-embeddability}, neither of which are true for the class of SIM sets.  We achieve these results by proving a simple nonstandard characterization of being supra-SIM.

In the final section, we continue the theme of proving analogues of results for positive Banach density with (supra-)SIM assumptions by considering results on sumsets.  Indeed, we prove the SIM analogue of Jin's sumset theorem (\cite{jin}) as well as Nathanson's result from \cite{nathanson},  which yielded partial progress on \emph{Erd\H os' $B+C$ conjecture} (which was recently solved in \cite{Moreira}).

We assume that the reader is familiar with basic nonstandard analysis as it pertains to combinatorial number theory.  Alternatively, one can consult the recent manuscript \cite{DiGoldLup}, which also contains a chapter on SIM sets.  Nevertheless, we will recall the relevant definitions and facts about SIM sets in the next section.

We thank Mauro Di Nasso for useful conversations regarding this work.

\section{Preliminaries}

Let $I:=[y,z]$ be an infinite, hyperfinite interval.  Set $\st_I:=\st_{[y,z]}:I\to [0,1]$ to be the map $\st_I(a):=\st(\frac{a-y}{z-y})$.  For $A\subseteq \starN$ internal, we set $\st_I(A):=\st_I(A\cap I)$.  We recall that $\st_I(A)$ is a closed subset of $[0,1]$ and we may thus consider $\lambda_I(A):=\lambda(\st_I(A))$, where $\lambda$ is Lebesgue measure on $[0,1]$.  

We also consider the quantity $g_A(I):=\frac{d-c}{|I|}$, where $[c,d]\subseteq I$ is maximal so that $[c,d]\cap A=\emptyset$.

The main idea in what is to follow is the desire to compare the notions of making $g_A(I)$ small (an internal notion) and making $\lambda_I(A)$ large (an external notion).  There is always a connection in one direction, namely that if $\lambda_I(A)> 1-\epsilon$, then $g_A(I)<\epsilon$.  We now consider sets where there is also a relationship in the other direction.

\begin{df}
We say that $A$ has the \emph{interval-measure property}\index{interval-measure property} (or \emph{IM property}) on $I$ if for every $\epsilon>0$, there is $\delta>0$ such that, for all infinite $J\subseteq I$ with $g_{A}(J)\leq \delta$, we have $\lambda_{J}(A)\geq 1-\epsilon$.
\end{df}

If $A$ has the IM property on $I$, we let $\delta(A,I,\epsilon)$ denote the supremum of the $\delta$'s that witness the conclusion of the definition for the given $\epsilon$.

It is clear from the definition that if $A$ has the IM property on an interval, then it has the IM property on every infinite subinterval.  Also note that it is possible that $A$ has the IM property on $I$ for a trivial reason, namely that there is $\delta>0$ such that $g_A(J)>\delta$ for every infinite $J\subseteq I$.  Let us temporarily say that $A$ has the \emph{nontrivial IM property} on $I$ if this does \emph{not} happen, that is, for every $\delta>0$, there is an infinite interval $J\subseteq I$ such that $g_A(J)\leq \delta$.  It will be useful to reformulate this in different terms.  In order to do that, we recall an important standard tool that is often employed in the study of sets with the IM property, namely the \emph{Lebesgue density theorem}.  Recall that for a measurable set $E\subseteq [0,1]$, a point $r\in E$ is a \emph{(one-sided) point of density of $E$} if
$$\lim_{s\to r^+}\frac{\mu(E\cap[r,s])}{s-r}=1.$$  The Lebesgue density theorem asserts that almost every point of $E$ is a density point of $E$.

\begin{fact}
Suppose that $A\subseteq \starN$ is internal and $I$ is an infinite, hyperfinite interval such that $A$ has the IM property on $I$.  Then the following are equivalent:
\begin{enumerate}
\item There is an infinite subinterval $J$ of $I$ such that $A$ has the nontrivial IM property on $J$.
\item There is an infinite subinterval $J$ of $I$ such that $\lambda_J(A)>0$.
\end{enumerate}
\end{fact}

In practice, the latter property in the previous proposition is easier to work with.  Consequently, let us say that $A$ has the \emph{enhanced IM property on $I$} if it has the IM property on $I$ and $\lambda_I(A)>0$.

In the proof of our main partition regularity result, the following \emph{internal} partition regularity theorem will be essential:  

\begin{thm}\label{IMpartreg}
Suppose that $A$ has the enhanced IM property on $I$.  Further suppose that $A\cap I=B_1\cup\cdots \cup B_n$ with each $B_i$ internal.  Then there is $i$ and infinite $J\subseteq I$ such that $B_i$ has the enhanced IM property on $J$.
\end{thm}

\begin{proof}
We prove the theorem by induction on $n$.  The result is clear for $n=1$.  Now suppose that the result is true for $n-1$ and suppose $A\cap I=B_1\cup\cdots \cup B_n$ with each $B_i$ internal.  If there is an $i$ and infinite $J\subseteq I$ such that $B_i\cap J=\emptyset$ and $\lambda_J(A)>0$, then we are done by induction.  We may thus assume that whenever $\lambda_J(A)>0$, then each $B_i\cap J\not=\emptyset$.  We claim that this implies that each of the $B_i$ have the IM property on $I$.  Since there must be an $i$ such that $\lambda_I(B_i)>0$, for such an $i$ it follows that $B_i$ has the enhanced IM property on $I$.  

Fix $i$ and set $B:=B_i$.  Suppose that $J\subseteq I$ is infinite, $\epsilon>0$, and $g_B(J)\leq \delta(A,I,\epsilon)$; we show that $\lambda_J(B)\geq 1-\epsilon$.  Since $g_A(J)\leq g_B(J)\leq \delta(A,I,\epsilon)$, we have that $\lambda_J(A)\geq 1-\epsilon$.  Suppose that $[r,s]\subseteq [0,1]\setminus \st_J(B)$.  Then $r=\st_J(x)$ and $s=\st_J(y)$ with $\frac{y-x}{|J|}\approx s-r$ and $B\cap[x,y]=\emptyset$.  By our standing assumption, this implies that $\lambda_{[x,y]}(A)=0$, whence it follows that $\lambda_J(A\cap [x,y])=0$.  It follows that $\lambda_J(B)=\lambda_J(A)\geq 1-\epsilon$, as desired.
\end{proof}

We will need two other facts about SIM sets, both of which are implicit in \cite{Leth} but are spelled out in more detail in \cite{DiGoldLup}:

\begin{fact}\label{comparable}
If $A$ is an internal set that has the IM property on $I$, then there is $w\in \n$ and a descending hyperfinite sequence $I=I_0,I_1,\ldots,I_K$ of hyperfinite subintervals of $I$ such that:
\begin{itemize}
\item $|I_K|\leq w$;
\item $\frac{|I_{k+1}|}{|I_k}|\geq \frac{1}{w}$ for all $k<K$;
\item whenever $I_k$ is infinite, we have $\lambda_{I_k}(A)>0$.
\end{itemize}
\end{fact}

\begin{fact}\label{corIMmain}
Suppose that $A_1,\ldots,A_n$ are internal sets that satisfy the IM property on $I_1,\ldots,I_n$ respectively.  Fix $\epsilon>0$ such that $\epsilon<\frac{1}{n}$.  Take $\delta>0$ with $\delta<\min_{i=1,\ldots,n}\delta(A_i,I_i,\epsilon)$.  Then there is $w\in \n$ such that, whenever $[a_i,a_i+b]$ satisfies $$[a_i,a_i+b]\subseteq I_i \text{ and }g_{A_i}([a_i,a_i+b])\leq \delta \text{ for all }i=1,\ldots,n,$$ then there is $c\in \starN$ such that
$$A_i\cap [a_i+c,a_i+c+w]\not=\emptyset \text{ for all }i=1,\ldots,n.$$
\end{fact}

%
%
%
%

We finally recall the definition of SIM sets:
\begin{df}
$A\subseteq \n$ has the \emph{standard interval-measure property}\index{standard interval-measure property} (or \emph{SIM property}) if:
\begin{itemize}
\item $\starA$ has the IM property on every infinite hyperfinite interval; 
\item $\starA$ has the enhanced IM property on some infinite hyperfinite interval.
\end{itemize}
\end{df}

It is possible to give a reformulation of SIM sets in completely standard terms; see \cite{Leth} for the details.  
\section{Supra-SIM sets and their properties}

We begin by noting that the collection of SIM sets is not closed under the operation of taking supersets.
\begin{ex}
Suppose that $B$ has the SIM property but is not syndetic.  Then as shown in \cite{Leth}, there is $A\supseteq B$ such that $A$ is not SIM.
\end{ex}
This implies that not all piecewise syndetic sets are SIM sets.  As we will see below, the property of being a SIM set is also not partition regular.  It is thus more interesting to consider the notion of a ``supra-SIM'' set, defined below.

\begin{df}
$A\subseteq \n$ is \emph{supra-SIM} if there is $B\subseteq A$ such that $B$ has the SIM property.
\end{df}

\begin{ex}[\cite{Leth}]
Piecewise syndetic sets are supra-SIM.
\end{ex}
In \cite{Leth}, SIM sets of Banach density $0$ are constructed.  This implies that there are supra-SIM sets that do not have positive Banach density, and thus also are not piecewise syndetic.

In order to prove our main results on supra-SIM sets, we use a convenient nonstandard reformulation.  The next theorem is the core of the matter:

\begin{thm}\label{suprachar}
Suppose that $A\subseteq \n$ is such that $\starA$ has the enhanced IM property on some interval $I$.  Then $A$ is a supra-SIM set.
\end{thm}

\begin{proof}
Without loss of generality, $I\subseteq {}^{\ast}\n\setminus \n$.  For each $\epsilon>0$, fix $\delta(\epsilon)<\min(\delta(A,I,\epsilon),\epsilon,\frac{1}{4})$.
For ease of notation, we set $\delta_k:=\delta(\frac{1}{k})$.  By underflow, for each $n,k\in \n$, there exists $M_{n,k}\in \n$ such
that whenever a subinterval $J$ of $I$ satisfies $g_{^{\ast}A}(J)<\delta_k$ and $l(J)>M_{n,k}$, then it takes the sum of the lengths of at least $n$
gaps of $\starA$ on $J$ to add to $
\frac{l(J)}{k}$. \ Since $\lambda_{I}($ $^{\ast}A)>0$, for each $n$
there exists an infinite subinterval $J$ of $I$ such that $g_{\starA}(J)<\frac{1}{n}$. \ 

By transfer, we may inductively define a sequence of pairwise disjoint
intervals $\left(  I_{n}\right)  $ in $\mathbb{N}$ satisfying the following properties:
\begin{enumerate}
\item[(i)] Writing $I_{n}=[a_{n},b_{n}]$, we have $a_{n}>nb_{n-1}$.
\item[(ii)] $I_{n}$ has a subinterval of length at least $n$ with $g_{A}(J)<\frac
{1}{n}$.
\item[(iii)] For all $k\leq n$ and for all $J\subseteq I_{n}$, if $\left\vert
J\right\vert >M_{n,k}$ and $g_{A}(J)<\delta_k$, then at least $n$
gaps of $A$ on $J$ are required to cover at least $\frac{l(J)}{k}$. 
\end{enumerate}
Set $B:=\bigcup_{n}(A\cap I_{n})$.  We claim that $B$ has the SIM property.

Let $I^{\prime}$ be an infinite hyperfinite interval. \ We show that $^{\ast
}B$ has the IM property on $I^{\prime}$ as witnessed by the function $\delta'(\epsilon):=\frac{1}{2}\delta_k$, where $\frac{1}{k}<\epsilon$. \ Fix $\epsilon>0$ and consider an infinite subinterval $J$ of $I^{\prime}$ such that $g_{^{\ast}B}(J)\leq \frac{1}{2}\delta_k$. \ 

By condition (i), If $J$ intersects more than one of the $I_{K}$, with the
largest such index being $M$, then every point in any $J\cap I_{K}$ with $K<M$
is less than $\frac{1}{M}a_{M}$, and so is infinitesimal compared to the
length of $J$ (which is at least $a_{M}-b_{M-1}$). \ Thus, all these points
are mapped to 0 by the $\st_{J}$ mapping.  Next note that $a_M$ must be within the first $\delta_k$ portion of $J$, else $g_{\starB}(J)\geq \delta_k$.  If the right endpoint of $J$ is at most $b_M$, we then have that $l(J\cap I_M)\geq (1-\delta_k)l(J)$.  If $J$ ends after $I_M$, then again we see that $b_M$ must occur in the last $\delta_k$ portion of $J$, so  $l(J\cap I_M)\geq (1-2\delta_k)l(J)$.  In either case, we have $l(J\cap I_M)\geq (1-2\delta_k)l(J)$.

%
%
It follows that
$$g_{^{\ast}B}(J\cap I_{M})\leq g_{\starB}(J)\cdot \frac{l(J)}{l(J\cap I_M)}\leq \frac{\delta_k}{2(1-2\delta_k)}\leq \delta_k.$$
Since $g_{^{\ast}B}(J\cap I_{M})=g_{^{\ast}A}(J\cap I_{M})$ and it requires
$M$ gaps of $\starA$ to add to $\frac{l(J)}{k}$, we see that $\lambda_J(^{\ast
}B)\geq 1-\frac{1}{k}>(1-\epsilon)$, as desired. \

It remains to show that $\starB$ has the enhanced IM property on some interval.  To see that, observe that if $N>\n$, then $I_N$ has a subinterval $J$ of size at least $N$ with $g_{\starA}(J)\leq \delta_N\approx 0$; since $g_{\starB}(J)=g_{\starA}(J)$, we see that $\starB$ has the enhanced IM property on $J$.
\end{proof}

Here is our promised nonstandard reformulation of supra-SIM sets:

\begin{cor}
$A$ is supra-SIM if and only if there is $B\subseteq A$ and infinite hyperfinite $I$ such that $\starB$ has the enhanced IM property on $I$.
\end{cor}

\begin{proof}
If $A$ is supra-SIM, then there is $B\subseteq A$ that is SIM.  By definition of SIM, this $B$ is as desired.  Conversely, if $B$ and $I$ are as in the condition, then $B$ is supra-SIM by the theorem, whence so is $A$.
\end{proof}


The partition regularity of supra-SIM now follows easily:

\begin{cor}\label{PR}
The notion of being a supra-SIM set is partition regular.
\end{cor}

\begin{proof}
Suppose that $A$ is supra-SIM and $A=C\sqcup D$.  Take $B\subseteq A$ SIM.  Take infinite $I$ such that $\starB$ has the enhanced IM property on $I$.  Then by Theorem \ref{IMpartreg}, we have, without loss of generality, that ${}^{\ast}(B\cap C)$ has the enhanced IM property on some infinite subinterval of $I$.  It follows from the previous corollary that $C$ is supra-SIM.
\end{proof}

\begin{cor}
Every supra-SIM set is contained in an ultrafilter consisting entirely of supra-SIM sets.
\end{cor}

\begin{ex}
Being SIM is not partition regular.  Indeed, consider
$$A:=\{1,3,4,7,8,9,13,14,15,16,\ldots\},$$ where $A$ continues to consist of $m$ elements in the set followed by $m$ elements that are not in the set, with $m$ increasing by 1 each time.  Then, if $k$ is large (but finite), on any infinite hyperfinite interval $I$ that consists of $k$ disjoint intervals that are in $\starA$ and $k$ disjoint intervals that are not in $\starA$, we have that $g_{\starA}(I)$ and $g_{{}^{\ast}(\mathbb N\setminus A)}(I)$ are both roughly equal to $1/(2k)$, while $\lambda_I(\starA)$ and $\lambda_I({}^{\ast}(\mathbb N\setminus A))$ are both $1/2$.

\end{ex}

The argument in the proof of Theorem \ref{suprachar} is robust enough to allow us to adapt it to prove another desirable property of supra-SIM sets that is also possessed by sets of positive Banach density.  Recall that $A$ is said to be \textit{finitely-embedded} in $B$ if, given any finite $F\subseteq A$, there is $t\in \n$ such that $t+F\subseteq B$.  (Equivalently, there is $t\in {}^{\ast}\n$ such that $t+A\subseteq \starB$.)  Note that if $A$ is finitely embedded in $B$ and $\BD(A)>0$, then $\BD(B)>0$.
  
\begin{thm}\label{finemb}
Suppose that $A$ is finitely embedded in $B$ and $A$ is supra-SIM.  Then $B$ is supra-SIM.
\end{thm}

\begin{proof}
Without loss of generality, we may assume that $A$ is actually SIM.  For $n\in \n$, let $X_n$ be the set of intervals $I$ in ${}^{\ast}\n$ of length at least $n$ such that $g_{\starA}(I)\leq \frac{1}{n}$ and $t+(\starA\cap I)\subseteq \starB$ for some $t\in {}^{\ast}\n$.  Since $A$ is SIM and finitely-embeddable in $B$, each $X_n\not=\emptyset$.  Thus, by overflow, there is $I\in \bigcap_n X_n$.

As in the proof of Theorem \ref{suprachar}, we may use transfer to inductively define a sequence of pairwise disjoint
intervals $\left(  I_{n}\right)  $ in $\mathbb{N}$ and a sequence $(t_n)$ from $\n$ satisfying the following properties:
\begin{enumerate}
\item[(i)] Writing $t_n+I_{n}=[a_{n},b_{n}]$, we have $a_{n}>nb_{n-1}$.
\item[(ii)] $I_{n}$ has a subinterval of length at least $n$ with $g_{A}(J)<\frac
{1}{n}$.
\item[(iii)] For all $k\leq n$ and for all $J\subseteq I_{n}$, if $\left\vert
J\right\vert >M_{n,k}$ and $g_{A}(J)<\delta_k$, then at least $n$
gaps of $A$ on $J$ are required to cover at least $\frac{l(J)}{k}$.
\item[(iv)] $t_n+(A\cap I_n)\subseteq B$.
\end{enumerate}

Let $C:=\bigcup_n (t_n+(A\cap I_n))$.  As in the proof of Theorem \ref{suprachar}, $C$ has the SIM property.  By (iv), $C\subseteq B$, so $B$ is supra-SIM, as desired.
\end{proof}

Of course the previous proposition fails for SIM sets for, as mentioned in the introduction, they are almost never even closed under taking supersets.

We end this section by mentioning arguably the most pressing open question concerning supra-SIM sets:

\begin{question}
Are sets of positive Banach density supra-SIM?
\end{question}

Our results from this section yield a \emph{prima facie} simpler criterion for obtaining a positive solution to the previous question.  First recall that, for $A\subseteq \n$, the \emph{Shnirelmann density} of $A$ is $\sigma(A):=\inf_{n\geq 1}\frac{|A\cap [1,n]|}{n}$.

\begin{cor}
Suppose there is $\epsilon>0$ such that every set $A\subseteq \n$ with $\sigma(A)\geq 1-\epsilon$ is supra-SIM.  Then every set of positive Banach density is supra-SIM.
\end{cor}

\begin{proof}
Suppose that $\epsilon$ is as in the hypothesis of the corollary and suppose that $\BD(A)>0$.  Take a finite $F\subseteq \n$ such that $\BD(A+F)\geq 1-\epsilon$.  Take $B\subseteq \n$ such that $B$ is finitely embedded in $A$ and $\sigma(B)\geq \BD(A+F)$ (see, for example, \cite[Corollary 12.12]{DiGoldLup}).  By assumption, $B$ is supra-SIM.  By Theorem \ref{finemb}, $A+F$ is supra-SIM.  By Corollary \ref{PR}, $A+i$ is supra-SIM for some $i\in F$.  It remains to observe that being supra-SIM is translation invariant.
\end{proof}

\section{SIMsets and sumsets}

\subsection{The sumset phenomenon}

One of the first successes of nonstandard methods in combinatorial number theory was the following theorem of Renling Jin:

\begin{fact}
Suppose that $A,B\subseteq \n$ are such that $\BD(A),\BD(B)>0$.  Then $A+B$ is piecewise syndetic.
\end{fact}

In this subsection, we prove the analogous result, replacing the positive Banach density assumption with a SIM assumption:

\begin{prop}
If $A$ and $B$ have the SIM property, then $A+B$ is piecewise syndetic.
\end{prop}

\begin{proof}
By Fact \ref{comparable} and the
Lebesgue density theorem, we can obtain intervals $I$ and $J$ of the same infinite length such that $\lambda_{I}($ $^{\ast}A)=\lambda_{J}($ $^{\ast}B)=1$. Fix $a\in ^{\ast}A\cap I$ and $b\in ^{\ast}B\cap
J$. \ Let $w\in$ $\mathbb{N}$ be as in Fact \ref{corIMmain} for $A_{1}:=$ $^{\ast
}A-a$, $A_{2}:=$ $b-$ $^{\ast}B$, $I_{1}:=I-a$, \ and $I_{2}:=J-b$. \ Then for any
finite $m$, $g_{A_{1}}(I_{1}+m)\approx0$ and $g_{A_{2}}(I_{2})\approx0$.
\ Thus, by the choice of $w$, there must exist $c\in$ $^{\ast}\mathbb{N}$ such
that
\[
A_{1}\cap\lbrack m+c,m+c+w]\not=\emptyset \text{ and }A_{2}\cap\lbrack c,c+w]\not=\emptyset.%
\]
If we fix $x\in A_{1}\cap\lbrack m+c,m+c+w]$ and $y\in A_{2}\cap\lbrack
c,c+w]$, then
\begin{align*}
x-y  & \in\left(  A_{1}-A_{2}\right)  \cap\lbrack m-w,m+w]\\
& =\text{ }\left(  \left(  ^{\ast}A-a\right)  -\left(  b-^{\ast}B\right)
\right)  \cap\lbrack m-w,m+w].
\end{align*}
This shows that there is an element of $^{\ast}A+$ $^{\ast}B$ in every
interval of the form $[a+b+m-w,a+b+m+w]$. \ By overspill, there is an infinite interval starting at $a+b$ in which there is no gap of $^{\ast}A+$
$^{\ast}B$ greater than $2w$, completing the proof.
\end{proof}

\subsection{Towards $B+C$ for SIMsets}

In \cite{erdos}, Erd\H os made the following conjecture:

\begin{conj}
Suppose that $A\subseteq \n$ is such that $\underline{d}(A)>0$.  Then there are infinite sets $B,C\subseteq \n$ such that $B+C\subseteq\n$.
\end{conj}

The first progress on this conjecture was due to Nathanson \cite{nathanson}:

\begin{fact}
Suppose that $\BD(A)>0$.  Then for any $n\in \n$, there are $B,C\subseteq \n$ such that $B$ is infinite, $|C|= n$, and $B+C\subseteq A$.
\end{fact}

Nathanson's result follows immediately from repeated applications of the following fact, which he attributes to Kazhdan in \cite{nathanson}:

\begin{fact}
Suppose that $\BD(A)>0$.  Then there are arbitrarily large $t\in \n$ such that $\BD(A\cap (A-t))>0$.
\end{fact}

We remark in passing that the proof of Kazhdan's lemma appearing in \cite{nathanson} is quite complicated but that it is possible to give a very simple nonstandard proof as in \cite{DiGoldLup}.

In this subsection, we prove the supra-SIM version of Nathanson's result.  First, we should mention that, building somewhat upon ideas from \cite{canada}, Moreira, Richter, and Robertson positively settle the Erd\H os conjecture in \cite{Moreira}, even weakening the hypothesis to positive Banach density and also proving a version for countable amenable groups.

Here is the Kazhdan lemma for supra-SIM sets:

\begin{prop}
(\bigskip Kazhdan Lemma for supra-SIM sets) Suppose that $A\subseteq \n$ is supra-SIM and set $\mathcal{T}_A:=\{t\in \n \ : \ A\cap (A-t) \text{ is supra-SIM}\}$.  Then $\mathcal{T}_A$ is syndetic.

\end{prop}

\begin{proof}
%
%
Suppose that $^{\ast}A$ has the enhanced IM\ property on the interval $I$.  Let $w\in\mathbb{N}$ be as in Fact \ref{corIMmain} for $A_{1}:=A_{2}:=$ $^{\ast}A$ and $I_{1}:=I_{2}:=I,$ for some
appropriately small $\varepsilon$ and corresponding $\delta$. \ We show that $\mathcal{T}_A$ has no gaps of length larger than $w$.  Towards this end, fix $t\in \n$ and set%
\[
B_{t}:=\bigcup_{k=0}^{w}\left(  \text{ }^{\ast}A-(t+k)\right)  \text{.}%
\]

\noindent \textbf{Claim:}  If $J$ is any subinterval of $I$ on which $\lambda_{J}($ $^{\ast}A)>0,$then
\[
^{\ast}A\cap B_{t}\cap J\neq\varnothing.
\]

\

\noindent \textbf{Proof of the Claim:}  By the Lebesgue density theorem, we may choose $[a_{1},b]\subset J$ with sufficiently small gap that
we may apply Fact \ref{corIMmain} with $a_{2}:=a_{1}+t$.  This allows us to find a $c$ with $c+w\leq b$ such that
\begin{align*}
^{\ast}A\cap\lbrack a_{1}+c,a_{1}+c+w]  & \neq\varnothing\\
^{\ast}A\cap\lbrack a_{1}+t+c,a_{1}+t+c+w]  & \neq\varnothing.
\end{align*}
This is equivalent to:
\begin{align*}
^{\ast}A\cap\lbrack a_{1}+c,a_{1}+c+w]  & \neq\varnothing\\
\left(  ^{\ast}A-t\right)  \cap\lbrack a_{1}+c,a_{1}+c+w]  & \neq\varnothing.
\end{align*}

Let $d$ be an element in $^{\ast}A\cap\lbrack a_{1}+c,a_{1}+c+w]$. $\ $That
same $d$ must then be in $B_{t}$ since it is within $w$ of an element
in\ $\left(  ^{\ast}A-t\right)  $, and this completes the proof of the claim.

The claim implies that, for any infinite subinterval $J$ of $I$, we have that
\[
\lambda_{J}(\text{ }^{\ast}A\cap B_{t})=\lambda_{J}(\text{ }^{\ast}A),
\]
as $J$ cannot contain any infinite intervals in the complement of $^{\ast
}A\cap B_{t}$ that have positive $^{\ast}A$ measure.  It follows immediately that $^{\ast}A\cap B_{t}$ has the enhanced IM property on $I$.  By Theorem \ref{IMpartreg}, it follows that for some $k=0,...,w$, we have that
\[
^{\ast}A\cap(^{\ast}A-(t+k))
\]
has the enhanced IM property on some infinite subinterval of $I$.  For this $k$, it follows that $A\cap (A-(t+k))$ is supra-SIM.

%

\end{proof}

As in the case of the original Nathanson result, repeated application of the previous proposition implies:

\begin{cor}  (Nathanson's theorem for supra-SIM sets)
Suppose that $A$ is supra-SIM.  Then for any $n\in \n$, there is an infinite $B\subseteq A$ and $C\subseteq \n$ with $|C|=n$ such that $B+C\subseteq A$.
\end{cor}

Of course, we should ask:

\begin{question}
Suppose that $A$ is supra-SIM.  Do there exist infinite $B,C\subseteq \n$ such that $B+C\subseteq A$?
\end{question}

\end{document}